\documentclass[12pt,twoside]{elsarticle}
\usepackage[a4paper]{geometry}
\usepackage{amsmath,amsfonts,latexsym, amsthm, amssymb, mathabx, enumerate, tikz-cd, bm}
\usepackage{comment}
\usepackage[utf8]{inputenc}

\DeclareMathOperator{\Tr}{Tr}
\DeclareMathOperator{\Col}{Col}
\DeclareMathOperator{\Aut}{Aut}
\DeclareMathOperator{\GL}{GL}
\DeclareMathOperator{\Span}{Span}
\DeclareMathOperator{\Aff}{Aff}
\DeclareMathOperator{\Conv}{Conv}

\newcommand{\scalar}[2]{\langle #1 \, |\, #2 \rangle}

\newcommand{\Z}{\mathbb{Z}}
\newcommand{\Q}{\mathbb{Q}}
\newcommand{\R}{\mathbb{R}}
\newcommand{\C}{\mathbb{C}}
\newcommand{\CCC}{\mathbb{C}}
\newcommand{\F}{\mathbb{F}}

\newcommand{\1}{\mathbf{1}}
\newcommand{\0}{\mathbf{0}}
\newcommand{\uu}{\mathbf{u}}
\newcommand{\zz}{\mathbf{z}}
\newcommand{\pp}{\mathbf{p}}
\newcommand{\ff}{\mathbf{f}}
\newcommand{\vv}{\mathbf{v}}
\newcommand{\sss}{\mathbf{s}}
\newcommand{\ww}{\mathbf{w}}
\newcommand{\yy}{\mathbf{y}}
\newcommand{\xx}{\mathbf{x}}
\newcommand{\dd}{\mathbf{e}}
\newcommand{\mm}{\bm{\mu}}

\newcommand{\U}{\mathbf{U}}
\newcommand{\CC}{\mathbf{C}}

\newcommand{\PP}{\mathbf{P}}
\newcommand{\M}{\mathbf{M}}
\newcommand{\Y}{\mathbf{Y}}
\newcommand{\J}{\mathbf{J}}

\newcommand{\ZMM}{\mathbb{Z}^{\times}_{\ell/d}}
\newcommand{\Od}{\mathcal{O}_d}
\newcommand{\fld}{\frac{\ell}{d}}

\newcommand{\Ztimes}{\mathbb{Z}_\ell\rtimes\mathbb{Z}^{\times}_\ell}

\newcommand{\ZZtimes}{(\mathbb{Z}_\ell\times\mathbb{Z}_\ell)\rtimes \mathbb{Z}^{\times}_\ell}

\numberwithin{equation}{section}

\newtheorem{theorem}{Theorem}[section]
\newtheorem{corollary}[theorem]{Corollary}

\newtheorem{conjecture}[theorem]{Conjecture}
\newtheorem{lemma}[theorem]{Lemma}
\newcommand{\etal}{\emph{et al.}}

\begin{document}
\begin{frontmatter}
	\title{The dimension of an orbitope based on a solution to the Legendre pair problem}
	\author[AFIT]{Kristopher N.~Kilpatrick}
\ead{kristopher.n.kilpatrick@gmail.com}	
	\author[AFIT]{Dursun A.~Bulutoglu}
\ead{dursun.bulutoglu@gmail.com}
\address[AFIT]{Department of Mathematics and Statistics, Air Force Institute of Technology,\\Wright-Patterson Air Force Base, Ohio 45433, USA}

\begin{abstract}
The Legendre pair problem is a particular case of a rank-$1$ semidefinite description problem that seeks to find a pair of vectors $(\uu,\vv)$ each of length $\ell$ such that the vector $(\uu^{\top},\vv^{\top})^{\top}$ satisfies the rank-$1$ semidefinite description.
The group $\ZZtimes$ acts on the solutions satisfying the  rank-$1$ semidefinite description  by 
 $
 ((i,j),k)(\uu,\vv)=((i,k)\uu,(j,k)\vv)
 $
 for each $((i,j),k) \in \ZZtimes$.
By applying the methods based on representation theory in Bulutoglu [Discrete Optim. 45 (2022)], and results in Ingleton [Journal of the London Mathematical Society s(1-31) (1956), 445-460] and Lam and Leung [Journal of Algebra 224 (2000), 91-109],  for a given solution $(\uu^{\top},\vv^{\top})^{\top}$ satisfying the  rank-$1$ semidefinite description, we  show that 
 the dimension of the convex hull of the orbit of $\uu$ under the action of $\Z_{\ell}$ or $\Ztimes$ is $\ell-1$ provided that $\ell=p^n$ or $\ell=pq^i$ for $i=1,2$,  any positive integer $n$, and any two odd primes $p,q$. Our results lead to the conjecture that this dimension is $\ell-1$ in both cases. We also show that   the dimension of the convex hull of all feasible points of the Legendre 
pair problem of length $\ell$ is $2\ell-2$ provided that it has at least one feasible point.
\end{abstract}
\end{frontmatter}
\section{Introduction}

A Hadamard matrix can be constructed by finding a solution to a system of constraints for a pair of vectors.  To define this system of constraints, let $\Z_\ell=\{0,\ldots, \ell-1\}$ denote the integers mod $\ell$.   Two length $\ell$ vectors $(\uu,\vv)$ is a \textit{Legendre pair} (LP) if 
\begin{equation}\label{eqn:PAFconstraint}
\begin{split}
P_\uu(j)+P_\vv(j)=-2, \, \, \, \forall j\in \Z_\ell-\{0\}, \\
\1^\top \uu=\1^\top \vv, \, \, \, \uu,\vv\in \{-1,1\}^\ell,
\end{split}
\end{equation}
where $P_\xx(j)=\sum_{i\in \Z_\ell} \xx(i)\xx(i-j)$ is the \textit{periodic autocorrelation function} of a vector $\xx\in \mathbb{C}^{\ell}$.  The problem of finding solutions to the system of constraints~(\ref{eqn:PAFconstraint}) is known as the \textit{LP problem}.
 System of constraints~(\ref{eqn:PAFconstraint}) can be cast as the following rank-$1$ constrained
 semidefinite description (rank-$1$ CSDD).  
 \begin{eqnarray}
\sum_{i-j=r\; (\text{mod } \ell) \atop 1\leq i,j\leq \ell}Y_{ij}+ \sum_{i-j=r\; (\text{mod } \ell) \atop \ell+1 \leq i,j\leq 2\ell}Y_{ij}&=&-2, \quad \text{for} \; r= 1, \ldots, \ell,  \nonumber \\
Y_{jj}& = &1, \quad \,\,\, \,\, \text{for} \; j= 1, \ldots, 2\ell,  \nonumber\\
\Y &\succeq & 0, \nonumber \label{eqn:SDP1} \\
 \text{rank}(\Y)& =& 1, \label{eqn:SDP2} \\
  \sum_{j=0}^{\ell-1} Y_{1j+1}&=&\sum_{j=0}^{\ell-1} Y_{1(\ell+j+1)}, \nonumber
\end{eqnarray}
where $\Y \succeq  0$ means  $\Y$ is non-negative definite.  
In this rank-$1$ CSDD for LPs, an LP $(\uu,\vv)$ is obtained by  taking  $u_j=Y_{1(j+1)}$ and $v_j=Y_{1(\ell+j+1)}$ for $j=0,\ldots,\ell-1$. For the definition of the feasible set of a rank-$1$ constrained semidefinite programming problem, i.e. rank-$1$ CSDD, see~\cite{Dattorro2019}. 

Let $\Q^{\Z_\ell}$ be the vector space of all functions from $\Z_\ell$ to $\Q$.  A \textit{circulant shift} of $\uu\in \Q^{\Z_\ell}$ by $j\in \Z_\ell$, denoted by $c_j \uu$, is a transformation such that $c_j\uu(i)=\uu(i-j), i\in \Z_\ell$.  The \textit{circulant matrix} of $\uu\in \Q^{\Z_\ell}$, denoted by $\mathbf{C}_\uu$, is a matrix such that $(j+1)$th row of $\mathbf{C}_\uu$ is $(c_j\uu)^\top$.  If $(\uu,\vv)$ is an LP with $\1^\top \uu=\1^\top \vv=-1$, then 
\begin{equation*} 
\mathbf{H}=
\begin{bmatrix}
\phantom{-}1 & \phantom{-}1 & \phantom{-}\1^\top & \phantom{-}\1^\top \\
\phantom{-}1 & - 1 & \phantom{-}\1^\top & -\1^\top \\
\phantom{-}\1 & \phantom{-}\1 & \phantom{-}\mathbf{C}_\vv & \phantom{-}\mathbf{C}_\uu \\
\phantom{-}\1 & -\1 & \phantom{-}\mathbf{C}_\uu^\top & -\mathbf{C}_\vv^\top
\end{bmatrix}
\end{equation*} 
is a $(2\ell +2)\times (2 \ell +2)$ Hadamard matrix, where $\1$ is the vector of all $1$s of length $\ell$ \cite{arasu2020legendre}. A Hadamard matrix can also be constructed by using an LP with $\1^\top \uu=\1^\top \vv=1$. Hence, to construct a $(2\ell +2)\times (2 \ell +2)$ Hadamard matrix for some odd $\ell$, it suffices to find  an LP of length $\ell$. It is conjectured that an LP of length $\ell$ exists for each odd $\ell$, where this conjecture implies the Hadamard conjecture.  It is shown in Arasu \etal~\cite{arasu2020legendre} that an LP $(\uu,\vv)$ must satisfy 
\begin{equation*}
\1^\top \uu=\1^\top \vv=\pm 1.
\end{equation*}
In this paper, we choose all LPs $(\uu,\vv)$ to satisfy
\begin{equation}\label{eqn:feasible}
\1^\top \uu=\1^\top \vv=-1.  
\end{equation}


Let $N\rtimes H$ be the semidirect product of the two groups $N$ and $H$ as defined in Rotman~\cite{rotman2010advanced}.   
Then, the group $\Ztimes$ acts on $\uu \in \Q^{\Z_\ell}$ by 
$(j,k)\uu(i)=\uu((j,k)^{-1}i)$ for each  $(j,k)\in\Ztimes, i\in \Z_\ell$, where $(a,b)i=bi+a$ for $(a,b)\in\Ztimes, i \in \Z_\ell$.
 The group $\ZZtimes$ acts on any pair $(\uu,\vv)\in \Q^{\Z_\ell}\oplus \Q^{\Z_\ell}$ by 
 $$
 ((i,j),k)(\uu,\vv)=((i,k)\uu,(j,k)\vv)
 $$
 for each $ ((i,j),k) \in \ZZtimes$.
Two pairs $(\uu,\vv), (\uu',\vv')$ of vectors of length $\ell$ are \textit{equivalent} if either $(\uu',\vv')$ or $(\vv',\uu')$ is in the  orbit of $(\uu,\vv)$ under the action of $\ZZtimes$.  If $(\uu,\vv)$ is an LP and $(\uu',\vv')$ is equivalent to $(\uu,\vv)$, then $(\uu',\vv')$ is also an LP \cite{arasu2020legendre}. The {\em symmetry group} of a rank-constrained CSDD is the set of all permutations of its solution vector that send  solutions to solutions. Hence, the symmetry group of rank-$1$ CSDD for LPs contains the group $\ZZtimes$.

Throughout the paper,  
for each $ n \in \mathbb{Z}_{\geq 1}$, let $[n]=\{1,\ldots,n\}$, and  $$\{d : d \mid \ell\}=  \{d\in [\ell]: d\mid\ell\}.$$ 
For a set of vectors $S$ in a vector space over the field of scalars $\mathbb{F}$, Span$_{\mathbb F}(S$) is the span, Aff$_{\mathbb F}(S$) is the affine hull, and dim$_{\mathbb F}(S$) is the dimension  of the affine hull  of the vectors in $S$ over $\mathbb{F}$. For a matrix $\M$, $\Col_{\F}(\M)$ is the column space of $\M$ over $\F$.
If ${\mathbb F}$ is not provided, then ${\mathbb F}={\mathbb R}$. Also, let Conv($S$) be the convex hull of the vectors in $S$. For a finite group $G$ acting linearly on a vector space $V$, $G\,\vv$ is the orbit of $\vv \in V$ under the action of $G$, and   $\Conv(G\,\vv)$ is the {\em orbitope} of $G$ through $\vv$~\cite{Strumfels:2011}. 
Moreover, let $\mathcal{F}_{\ell}$ be the feasible set of all pairs $(\uu^{\top},\vv^{\top})^{\top}$ satisfying the rank-$1$ CSDD for the LP problem.  In this paper, we determine $\dim(\Conv(\mathcal{F}_{\ell}))$, and investigate all possible values of $\dim(\Conv((\Ztimes)\,\uu))$ and $\dim(\Conv((\Ztimes)\,\vv))$ when $(\uu,\vv)$ is an LP.  The following theorems and corollary are our main results.

\begin{theorem}\label{thm:main-1}
Let $\ell$ be an odd positive integer and $\mathcal{F}_{\ell}$ be the feasible set of  all pairs $(\uu^{\top},\vv^{\top})^{\top}$ satisfying  the rank-$1$ CSDD for the LP problem.  If $\mathcal{F}_{\ell}\neq \emptyset$, then
\begin{equation*}\label{eqn:Fu2}
 \dim(\Conv({\cal F}_{\ell}))=2\ell- 2.
 \end{equation*}
\end{theorem}
The following corollary follows immediately from  Theorem~\ref{thm:main-1}.
\begin{corollary}
 Let $(\uu^{\top},\vv^{\top})^{\top}$  satisfy the rank-$1$ CSDD for the LP problem. Then two linear constraints that are linearly independent from $\1^\top \uu=\1^\top \vv$ 
 implied by having
constraint
~(\ref{eqn:SDP2}) in the rank-$1$ CSDD for the LP problem are
$\1^\top \uu=\pm1$. Furthermore, no linear constraint  that is linearly independent from constraints $\1^\top \uu=\pm1$ and $\1^\top \uu=\1^\top \vv$ is implied.
\end{corollary}
\begin{proof} 
First, the rank-$1$ CSDD for the LP problem is feasible.
Then, any other linear constraint that is linearly independent  from the constraints $\1^\top \uu=\pm1$ and $\1^\top \uu=\1^\top\vv$ would necessarily imply that 
  $ \dim(\Conv({\cal F}_{\ell}))<2\ell- 2 $, contradicting Theorem~\ref{thm:main-1}. \qedhere
\end{proof}
\begin{theorem}\label{thm:main0}
Let $(\uu,\vv)$ be an LP, and either $G=\mathbb{Z}_{\ell}$ or $G=\Ztimes$.  Then there exists $U_1, U_2 \subseteq  \{d: d \mid \ell\}$ such that $U_1\cap U_2=\emptyset$ and
\begin{equation*}\label{eqn:Fu}
 \begin{array}{r}
 \dim(\Conv(G\,\uu))=\ell- 1-\left( \sum_{d \in U_1}\phi\left(\fld\right)\right), \\
 \dim(\Conv(G\,\vv))=\ell- 1-\left( \sum_{d \in U_2}\phi\left(\fld\right)\right).
 \end{array}
 \end{equation*}
\end{theorem}

\begin{theorem}\label{thm:main1}
Let $p,q$ be distinct odd primes and $n\in \Z_{\geq 1}$. Let $\ell=p^n$, or $\ell=pq^i$ for $i=1,2$.  Let $(\uu,\vv)$ be an LP of length $\ell$, and either $G=\mathbb{Z}_{\ell}$ or $G=\Ztimes$.  Then 
$$\dim(\Conv(G\,\uu))=
\dim(\Conv(G\,\vv))=\ell- 1.$$
\end{theorem}

 In Section~\ref{sec:background}, we present necessary background in representation theory, the power spectral density, vanishing sums of roots of unity, and affine geometry.  In Section~\ref{sec:mainresult}, we prove our main results and conjecture that the restrictions on $\ell$ in Theorem~\ref{thm:main1} other than $\ell$ being an odd positive integer can be dropped. Section~\ref{sec:discussion} is our discussion about future work.

\section{Background theory}\label{sec:background}
Let $G$ be a finite group and $V$ be a finite-dimensional vector space over a field $\mathbb{F}$.  Let $\GL(V)$ be the $\F$-automorphisms of $V$. 
An $n$-dimensional $\F$-\textit{representation} of $G$ is a pair $(\rho,V)$ such that $\rho\colon G\to \GL(V)$ is a homomorphism and $V$ is an $n$-dimensional vector space over $\F$.  A subspace $W$ of $V$ is an $\F$-\textit{subrepresentation} of $V$ if $\rho(g)W\subseteq W$ $\forall g\in G$.  A representation $(\rho,V)$ of $G$ is an \textit{irreducible} representation if the only subrepresentations of $V$ are Span$_{\F}(\0)$ and $V$.  Two $\mathbb{F}$-representations $\rho_1:G\rightarrow  \text{GL}({W})$ and $\rho_2:G\rightarrow  \text{GL}({W'})$ of $G$  
are {\em equivalent} if there is an invertible linear transformation $\phi:W \rightarrow W'\, \ni \, \phi(\rho_1(g)w)=\rho_2(g)\phi(w)$ $\forall w \in W$ and $g \in G$. For a decomposition $ V= m_1V_1 \obot \cdots \obot m_bV_b $ into irreducible representations of a group $G$, $m_i$ is the number of times the irreducible $\mathbb{C}$-representation $V_i$ appears up to equivalence. Such $m_i\geq 1$ is called the  {\em multiplicity} of  $V_i$.

The only fields $\F$ that we consider in this paper are $\Q, \R$, and $\C$, the rational, real, and complex numbers, respectively.  An $\mathbb{F}$-representation $\rho:G \rightarrow  \text{GL}({V})$ is {\em unitary} with respect to an inner product  $\left\langle\cdot,\cdot\right\rangle$ in $V$ if $\left\langle\rho(g)\vv,\rho(g)\uu\right\rangle=\langle\uu,\vv\rangle$ for all $\uu,\vv \in V$. Every representation is unitary with respect to a complex inner product~\cite{serre1977linear}.  We use the convention $\scalar{\alpha \uu}{\vv}=\bar{\alpha}\scalar{\uu}{\vv}$ for $\alpha\in \F$ and $\uu,\vv\in V$. 
If a representation $(\rho,V)$ of $G$ is such that $\rho(g)$ 
permutes  a basis of $V$ for all $g \in G$, then $(\rho,V)$ is called a {\em permutation $\F$-representation}. Every permutation $\F$-representation is unitary with respect to the complex dot product.
Throughout the paper, every group $G$ is finite, every vector space $V$ is finite-dimensional, and every representation of any specified group is a permutation representation and hence is unitary with respect to the complex dot product.  

\begin{theorem}\label{thm:Maschke}[Maschke]
Every  representation of a finite group is a direct sum of orthogonal irreducible subrepresentations with respect to some inner product.
\end{theorem}
The decomposition in Theorem \ref{thm:Maschke} is said to be \textit{multiplicity-free} if each irreducible appears only once. 
 The \textit{character} of an $\F$-representation $(\rho,V)$ of $G$ is the map $\chi_\rho\colon G\to \F$ defined by $\chi_\rho(g)=\Tr(\rho(g))$, where $\Tr(\rho(g))$ is the trace of $\rho(g)$.  We say that the character is an \textit{irreducible character} if the character belongs to an irreducible representation.  Two $\mathbb{F}$-representations of a finite group are equivalent if and only if they have the same character~\cite{Fassler}. 
 We may simply write the character as $\chi$ if the representation is clear from the context.
  The following theorem is from Serre~\cite{serre1977linear}.
\begin{theorem}\label{thm:projectionoperator}
Let $(\rho,V)$ be a $\C$-representation of a group $G$.  Let 

$$
V=m_1V_1\obot \ldots \obot m_hV_h
$$
with $m_i\in \Z_{\geq 1}$ be a decomposition of $V$ into irreducibles $(\rho_i,V_i)$ with characters $\chi_i$ $\forall i\in[h]$, where $\chi_i=\chi_{\rho_i}$.  Then the orthogonal projection $\PP_i$ of $V$ onto $m_iV_i=\bigoplus_{k=1}^{m_i}V_i$ is given by 
$$
\PP_i=\frac{\dim_\C(V_i)}{|G|}\sum_{g\in G}\overline{\chi_i(g)}\mathbf{M}_{\rho(g)}.
$$
where $\mathbf{M}_{\rho(g)}$ is the matrix of $\rho(g)$ in some orthonormal basis for $V$.
\end{theorem}

The \textit{exponent} $m$ of a group $G$ is the smallest nonnegative integer such that $g^m=e$ $\forall g\in G$.  Let $(\rho,V)$ be a $\C$-representation of a group $G$ with exponent $m$.  Then $\rho(g)^m=id$ $\forall g\in G$, where $id$ is the identity mapping on $V$.  Therefore, the eigenvalues of $\rho(g)$ are $m$th roots of unity.  
Throughout the paper, let $\zeta_m=e^{2\pi i/m}$.  Since $\chi(g)$ is the trace of $\rho(g)$, $\chi(g)\in \Q(\zeta_m)$, where $\Q(\zeta_m)$ is the field extension of $\Q$ obtained by adjoining $\zeta_m$.  
\newcommand{\Autqc}{{\rm Aut}(\mathbb{C})_{\mathbb{Q}(\zeta_m)}}
\newcommand{\Autc}{{\rm Aut}(\mathbb{C})}
\newcommand{\Autcc}{{\rm Aut} (\mathbb{Q}(\zeta_m))}
\newcommand{\Galz}{{\rm Gal}(\mathbb{Q}(\zeta_m)/\mathbb{Q})}
\newcommand{\Qz}{\mathbb{Q}(\zeta_m)} 

Let $\mathbb{Z}_m^{\times}=\{r\in \mathbb{Z}_m \mid (r,m)=1\}$ be the multiplicative group of $\mathbb{Z}_m$, where $(r,m)$ is the greatest common divisor of $r$ and $m$.
The automorphism group $\Autcc$ of $\Q(\zeta_m)$ is the Galois group of the field extension $\Q(\zeta_m)$ over  $\mathbb{Q}$ denoted by ${\rm Gal}(\Q(\zeta_m)/\Q)$. The elements of ${\rm Gal}(\Q(\zeta_m)/\Q)$ can be indexed by the elements of $\mathbb{Z}_m^{\times}$,
where for $\sigma_r\in {\rm Gal}(\Q(\zeta_m)/\Q) $ $$\sigma_r (\zeta_m)= \zeta_m^r, \,\,\sigma_r (q)=q\text{ for each $q \in \Q,$ and } r\in \mathbb{Z}_m^{\times}.$$ 
 The map $\sigma_r \rightarrow r$ is an isomorphism between 
${\rm Gal}(\Q(\zeta_m)/\Q)$ and $\mathbb{Z}_m^{\times}$. 
Let $\Autc$ be the group of all automorphisms of $\CCC$ 
and $$\Autqc=\{\tau_{|\Qz} \mid \tau\in \Autc \},$$ where $\tau_{|\Qz}$ is the
 restriction of $\tau \in \Autc$ to $\Qz$.   For each $\tau \in \Autc$, $ (\tau(\zeta_m))^m=\tau((\zeta_m)^m)=\tau(1)=1$, 
and $\tau(q)=q$ for each $q \in \Q$. Hence, we have $\Autqc\subseteq\Autcc$. 
Since any automorphism of a subfield of $\CCC$ can be extended to 
an automorphism of  $\CCC$~\cite{Yale1966}, each $\sigma_r \in \Autcc$ can be extended to a $\sigma'_r \in \Autc$. Then $\Autcc\subseteq\Autqc$ implying $$\Autcc = \Autqc. $$
Since the entries of the matrix of $\rho(g)$ are in $\Q( \zeta_m)$, $\Aut(\Q( \zeta_m))$  ($\Autc$) acts on the (characters of the)  $\C$-representations of $G$ of the same dimension.  


Let $V=\F^G$ be the space of all functions from $G$ to $\F$. Then the standard basis
$$
\dd_g(h)=\begin{cases}
			1 & \text{if $g=h$,}\\
            0 & \text{otherwise}
		 \end{cases}
$$
for $g,h \in G$ 
spans $V$.
We further equip $V$ with the usual inner product $\scalar{\cdot}{\cdot}$ (the complex dot product) so that $\{\dd_g\}_{g\in G}$ is an orthonormal basis.
Then the \textit{regular $\F$-representation} of $G$ is the pair $(\rho,V)$ such that $\rho\colon G\to \GL(V)$ is the  homomorphism determined by the action of 
$G$ on the basis of $V$, where $\rho(h)\dd_g=\dd_{hg}$ $\forall h,g\in G$. Clearly, the regular $\F$-representation of $G$ is also a permutation representation. 

The \textit{complexification} of a vector space $U$ over $\Q$ is defined as $U_\C=\C\otimes_\Q U$. It is clear that if $U$ is a $\Q$-representation, then $U_\C$ is the $\C$-representation obtained by extending the field of scalars $\Q$ of $U$ to $\C$. 
Let $\scalar{\cdot}{\cdot}$ be the inner product of $V$ over $\F$.  It is plain to show that an inner product on $V_\C$ may be defined as $\scalar{z\otimes \vv}{z'\otimes \vv'}=z\overline{z'}\scalar{\vv}{\vv'}$. Then the following lemma follows immediately. 
\begin{lemma}\label{lem:inner}
Let $V$ be an inner product space over $\Q$ and $U,W$ subspaces of $V$.  Then $U$ is orthogonal to $W$ if and only if $U_\C$ is orthogonal to $W_\C$.
\end{lemma}

The following theorem is well-known. 
\begin{theorem}\label{thm:Fourier}
Let $\left(R,\Q^{\Z_{\ell}}\right)$ be the regular $\Q$-representation of $\Z_\ell$.  
 Then $$\left(\Q^{\Z_\ell}\right)_\C=\bigobot_{i\in \Z_\ell}V_i$$
is the decomposition of $\left(\Q^{\Z_\ell}\right)_\C$ into the 
one-dimensional irreducible $\C$-representations of $\Z_\ell$,
where the  $\C$-representations $(\rho_k,V_k), k\in \Z_\ell$ are given by $\rho_k(j)\colon V_k\to V_k$,  $\vv\mapsto  \overline{\zeta}_{\ell}^{jk}\vv$,  and $\chi_k(j)=\Tr(\rho_k(j))=\overline{\zeta}_{\ell}^{jk}$ for each $k,j\in \Z_\ell$.
\end{theorem}
For each $k\in \Z_\ell$,  $\vv_k=\sum_{j \in \Z_\ell}\zeta^{jk}_\ell\dd_j$ is the basis vector for the one-dimensional $\C$-representation $V_k$ of $\Z_\ell$. It is well-known that $\{\vv_0,\ldots, \vv_{\ell-1}\}$ is a basis of $\C^{\Z_\ell}$. Throughout the paper, we call the vectors $\vv_0,\ldots, \vv_{\ell-1}$  the \textit{discrete Fourier basis} vectors.

The following theorem characterizes the orbits of the one-dimensional representations in Theorem~\ref{thm:Fourier} under the action of $\Autc$.
\begin{theorem}\label{thm:charorbits}
For $k\in \Z_\ell$, let $\chi_k$ be the irreducible character of the $\C$-representation of $\Z_\ell$.  For each divisor $d$ of $\ell$, let $\Od=\{\chi_k \,|\, (k,\ell)=d \text{ and } k\in \Z_\ell \}$.  Then 
\begin{itemize}
\item[(i)] $|\Od|=\phi(\ell/d)$, where $\phi$ is Euler's totient function.
\item[(ii)] $\Od=\Autc\,\chi_d=\Aut(\Q(\zeta_{\ell/d}))\,\chi_d$.
\end{itemize} 
\end{theorem}
\begin{proof}
Let $d$ be a divisor of $\ell$.  By the definition of Euler's totient function, (i) follows immediately.
To prove (ii), 
 $$\sigma_r(\chi_d(j))=\overline{\zeta}_{\ell}^{jrd}=\chi_{rd}(j)$$
for each $\sigma_r\in \Aut(\Q(\zeta_{\ell/d}))$, $r \in \mathbb{Z}_{\ell/d}^\times$, and $j \in \mathbb{Z}_{\ell}$. Hence, $\sigma_r(\chi_d)=\chi_{rd}$, 
and $(rd,\ell)=d$ implies $\chi_{rd}\in \Od$, consequently  $\Aut(\Q(\zeta_{\ell/d}))\,\chi_d \subseteq \Od$.

Let $\chi_k \in \Od$ such that $k=rd$ and $(r,\ell/d)=1$.  Then $\sigma_r\in \Aut(\Q(\zeta_{\ell/d}))$ and  
$$\chi_k=\chi_{rd}=\sigma_r \chi_d$$
Hence, $\Od \subseteq \Aut(\Q(\zeta_{\ell/d}))\,\chi_d $.
\end{proof}
 The following theorem follows from Isaacs~\cite [Theorem 9.21]{Isaacs}, 
  Theorem~\ref{thm:projectionoperator},
  and Lemma~\ref{lem:inner}.

\begin{theorem}\label{thm:bulutoglu}
Let $(\rho,V)$ be a $\Q$-representation of a group G with exponent $m$.  Let 
$$
V=W_1\obot \dots \obot W_b
$$ 
be a decomposition of $V$ into irreducible $\Q$-subrepresentations.  Let $V_\C, W_{i \C}$ be the $\C$-representations obtained from $V, W_i$ by extending the field of scalars of $V,W_i$ to $\C$.  Let
$$
V_\C=W_{(1,1)}\obot \dots \obot W_{(1,r_1)}\obot \dots \obot W_{(b,1)}\obot \dots \obot W_{(b,r_b)}
$$
be a decomposition of $V_\C$ into $\left(\rho_{i,j},W_{(i,j)}\right)$ irreducible $\C$-subrepresentations with characters $\chi_{\rho_{i,j}}$,
where 
$$
W_{i\C}=W_{(i,1)}\obot \dots \obot W_{(i,r_i)}\, \text{ for } \, i\in [b].
$$
For $i\in[b]$, let $\mathcal{O}_{\rho_{i,1}}$ be the $\Aut(\Q( \zeta_m))$-orbit of $\chi_{\rho_{i,1}}$.  Let $\mathbf{M}_{\rho(g)}$ be the matrix of $\rho(g)$ $\forall g\in G$, with respect to the standard basis.  If $V_\C$ is multiplicity-free, then 
\begin{equation*}
\mathbf{P}_{W_i}=\frac{\dim_\C(W_{(i,1)})}{|G|} \sum_{g\in G} \left(\sum_{\chi\in \mathcal{O}_{\rho_{i,1}}} \overline{\chi(g)}{\mathbf M}_{\rho(g)}\right)
\end{equation*}
is the orthogonal projection matrix into the $i$th irreducible $\Q$-subrepresentation subspace $\Col_\Q(\PP_{W_i})$ $\forall i\in [b]$.
\end{theorem}
The following theorem provides the decomposition of $\Q^{\Z_{\ell}}$ into irreducible $\Q$-subrepresentations under the action of $\Z_{\ell}$.
\begin{theorem}\label{thm:irreducibleQ}
Let $\left(R,\Q^{\Z_{\ell}}\right)$ be the regular $\Q$-representation of $\Z_\ell$.    Let $$\PP_d=\frac{1}{\ell} \sum_{i\in \Z_\ell}\sum_{\chi\in \Od} \overline{\chi(i)}\mathbf{M}_{R(i)}$$
for each divisor $d$ of $\ell$, where $\mathbf{M}_{R(i)}$ is the matrix of $R(i)$ with respect to the standard basis $\{\dd_i\}_{i\in \Z_\ell}$, and $\Od$ is defined in Theorem~\ref{thm:charorbits}.  Then 
\begin{equation*}\label{eqn:irreducibleQdecomposition}
\Q^{\Z_{\ell}}=\bigobot_{d\,|\,\ell}\Col_{\Q}(\PP_d)
\end{equation*}
is the decomposition of $\Q^{\Z_{\ell}}$ into irreducible $\Q$-subrepresentations under the action of $\Z_{\ell}$. 
\end{theorem}
\begin{proof}
Let $d$ be a divisor of $\ell$.  By Theorem~\ref{thm:charorbits}, the inner sum of $\PP_d$ is over the orbit of $\chi_d$ under the action of $\Aut(\Q(\zeta_{\ell/d}))$.  Further, by Theorem~\ref{thm:Fourier},  $\left(\Q^{\Z_\ell}\right)_\C$ is multiplicity-free. Then by Theorem~\ref{thm:bulutoglu}, $\Col_{\Q}(\PP_d)$ is an irreducible $\Q$-subrepresentation of $\Q^{\Z_{\ell}}$.  Now,   by the identity $\ell=\sum_{d\,|\,\ell}\phi(\ell/d)$, it suffices to show that $$\dim_\Q\left(\Col_{\Q}(\PP_d)\right)=
\phi\left(\frac{\ell}{d}\right).$$
Since $R$ is the regular representation of $\Z_\ell$, only $R(0)$ will contribute to the trace of $\PP_d$.  Since $\PP_d$ is a projection matrix,
\begin{align*}
\dim_\Q\left(\Col_{\Q}(\PP_d)\right)={\rm rank}{(\PP_d)}=\Tr{(\PP_d)} &=\Tr{\left(\frac{1}{\ell} \sum_{i\in \Z_\ell}\sum_{\chi\in \Od} \overline{\chi(i)}\mathbf{M}_{R(i)}\right)}\\
                    &=\frac{1}{\ell} \sum_{i\in \Z_\ell}\sum_{\chi\in \Od} \overline{\chi(i)}\Tr{\left(\mathbf{M}_{R(i)}\right)}\\
                    &=\frac{1}{\ell}\sum_{\chi\in \Od}  \Tr{\left(\mathbf{M}_{R(0)}\right)}\\
                    &=\frac{1}{\ell}\sum_{\chi\in \Od}  \ell\\
                    &=\sum_{\chi\in \Od}1\\
                    &=|\Od|\\
                    &=\phi\left(\fld\right). \qedhere
\end{align*} 
\end{proof}
 The group $\Ztimes$ acts on $\Z_\ell$  by $(a,b)i=bi+a$ $\forall i\in \Z_\ell$ and $(a,b)\in \Ztimes$. Then a representation $(T,\Q^{\Z_{\ell}})$ of $\Ztimes$ is given by the action on the basis $T(a,b)\dd_i=\dd_{bi+a}$. The next theorem shows that  $\Q^{\Z_{\ell}}$ has the same decomposition under the action of  $\Ztimes$ as in Theorem~\ref{thm:irreducibleQ}, i.e., extending  $\Z_\ell$ to  $\Ztimes$ does not change the decomposition in Theorem~\ref{thm:irreducibleQ}. 
\begin{theorem}\label{thm:Vsemirep}
The decomposition $\Q^{\Z_{\ell}}=\bigobot_{d\,|\,\ell} \Col_{\Q}(\PP_d)$ is the  decomposition of $\Q^{\Z_{\ell}}$ into irreducible $\Q$-subrepresentations under the action of $\Ztimes$.
\end{theorem}
\begin{proof}
Let $d$ be a divisor of $\ell$.  We have $\Col_{\C}(\PP_d)=\bigobot_{(k,\ell) = d} V_k$, where  $V_k$ is spanned by $\vv_k=\sum_{i\in \Z_\ell}\zeta_{\ell}^{k i}\dd_i$.  Let $(a,b)\in \Ztimes$.
Observe 
$$T(a,b)\vv_k=\sum_{i\in \Z_\ell}\zeta_{\ell}^{ki}\dd_{bi+a}=\sum_{j\in \Z_\ell}\zeta_{\ell}^{k(j-a)b^{-1}}\dd_j=\zeta_{\ell}^{-kab^{-1}}\vv_{b^{-1}k}.$$
Since $(b,\ell)=1$, $(bk,\ell)=d$ if and only if $(k,\ell)=d$.  Therefore, $\Col_{\C}(\PP_d)$ is a $\C$-subrepresentation under the action of $\Ztimes$.  Since $\Col_{\Q}(\PP_d)$ is an irreducible $\Q$-subrepresentation under the action of $\Z_\ell$ and $\Z_\ell$ is a subgroup of $\Ztimes$, $\Col_{\Q}(\PP_d)$ is an irreducible $\Q$-subrepresentation under the action of $\Ztimes$ .    
\end{proof}

  The \textit{discrete Fourier transform} (DFT) of $\uu\in \Q^{\Z_\ell}$ is $\mu_k(\uu)=\scalar{\uu}{\vv_{k}}, \, k\in \Z_\ell$. 
Equation $\mu_k(\uu)=\scalar{\uu}{\vv_{k}}, \, k\in \Z_\ell$ is the change of basis (variables) formula that provides the {\em Fourier coordinates}  $\mu_k(\uu)$ in terms of the {\em standard basis coordinates} $u_j$. 
The \textit{power spectral density} (PSD) of $\uu\in \Q^{\Z_\ell}$ is $|\mu_k(\uu)|^2,$ for $ k\in \Z_\ell$.  The following theorem provides an equivalent condition that characterizes an LP  in terms of its PSD.
\begin{theorem}[Fletcher \etal~\cite{Fletcher2001application}]\label{thm:PSD}
Let $\uu,\vv\in \{-1,1\}^\ell$.  Then $(\uu,\vv)$ is an LP if and only if  $\mu_0(\uu)=\mu_0(\vv)$ and 
$$
|\mu_k(\uu)|^2+|\mu_k(\vv)|^2=2(\ell+1), \quad \forall k\in \Z_\ell-\{0\}.
$$
\end{theorem}

We say that there is a \textit{vanishing sum of $m$ $\ell$th roots of unity}, if there exists $m$ $\ell$th roots of unity $w_1,\ldots, w_m$ (not necessarily distinct) satisfying $w_1+\dots +w_m=0$.

\begin{theorem}[Lam and Leung~\cite{lam2000vanishing}]\label{thm:Lam}
Suppose that $\ell=p_1^{n_1}\ldots p_s^{n_s}$ for distinct primes $p_1,\ldots, p_s$ and $n_1,\ldots, n_s\in \Z_{\geq 1}$.  Then there exists a vanishing sum of $m$ $\ell$th roots of unity if and only if $m=a_1p_1+\dots +a_sp_s$ for some $a_1,\ldots, a_s\in\Z_{\geq 0}$.  
\end{theorem}

Let $\uu\in \{-1,1\}^\ell$ satisfy $\scalar{\1}{\uu}=-1$.  Let $J=\{j \in \Z_\ell \,|\, \uu(j)=-1\}$.  Then $\mu_k(\uu)$  for $k\neq 0$,  has two forms 
\begin{equation}\label{eqn:lambdatwoforms}
\mu_k(\uu)=-2\sum_{j\in J}\zeta_{\ell}^{k j}=2\sum_{j\notin J}\zeta_{\ell}^{k j},
\end{equation}
where $|J|=(\ell+1)/2$.

\begin{lemma}\label{lem:vanishingprimepower}
Let $\ell=p^n$, $p$ an odd prime, and $n\in\Z_{\geq 1}$.  Let $\uu\in\{-1,1\}^\ell$ satisfy $\scalar{\1}{\uu}=-1$.  Then $\mu_k(\uu)\neq 0$ $\forall k\in \Z_\ell$.
\end{lemma}
\begin{proof}
Since $\mu_0(\uu)=-1$ we need to only verify for $k\in[\ell-1]$.  Suppose for a contradiction that $\mu_k(\uu)=0$ for some $k\in [\ell-1]$.  By equation~(\ref{eqn:lambdatwoforms}), $\sum_{j\in J}\zeta_{\ell}^{k j}=0$.  By Theorem~\ref{thm:Lam}, 
$$
\frac{\ell+1}{2}=ap
$$
for some $a\in \Z_{\geq 0}$.  This means that $p\,|\,((\ell+1)/2)$, a contradiction.  Therefore, $\mu_k(\uu)\neq 0$ $\forall k\in \Z_\ell$.    
\end{proof}

\begin{lemma}\label{lem:vanishingdistinctproductprime}
Let $\ell=pq$, $p,q$ distinct odd primes.  Let $\uu\in \{-1,1\}^\ell$ satisfy $\scalar{\1}{\uu}=-1$.  Then $\mu_k(\uu)\neq 0$ $\forall k\in \Z_\ell$.
\end{lemma}
\begin{proof}
Since $\mu_0(\uu)=-1$ we need to only verify for $k\in[\ell-1]$.  Suppose for a contradiction that $\mu_k(\uu)=0$ for some $k\in [\ell-1]$.  By equation~(\ref{eqn:lambdatwoforms}),
\begin{equation}\label{eqn:vanishing}
\sum_{j\in J}\zeta_{\ell}^{k j}=\sum_{j\notin J}\zeta_{\ell}^{k j}=0.
\end{equation}
By Theorem~\ref{thm:Lam}, 
$$
\frac{\ell+1}{2}=ap+bq \text{ and } \frac{\ell-1}{2}=cp+dq
$$
for some $a,b,c,d\in \Z_{\geq 0}$.  This means 
$$
\ell=\frac{\ell+1}{2}+\frac{\ell-1}{2}=(a+c)p+(b+d)q.
$$
If $a+c\neq0$ and $b+d\neq 0$, then $p\,|\,(b+d)$ and $q\,|\,(a+c)$.  Consequently,  
$$
\ell=(a+c)p+(b+d)q\geq pq+pq=2\ell,
$$
a contradiction.  Suppose that $a+c=0$.  Then $a=c=0$ and subsequently $q\,|\,((\ell+1)/2)$, a contradiction.  A similar contradiction occurs if $b+d=0$. 
Therefore, $\mu_k(\uu)\neq 0$ $\forall k\in \Z_\ell$.   
\end{proof}
Lemmas~\ref{lem:vanishingprimepower} and~\ref{lem:vanishingdistinctproductprime} do not generalize to arbitrary odd $\ell$. There is a $\uu\in \{-1,1\}^{45}$ with $\scalar{\1}{\uu}=-1$ such that $\mu_k(\uu)= 0$ for some $k\in \Z_{45}-\{0\}$.
 
The following  lemma is well-known.
\begin{lemma}\label{lem:zeroinaffine}
Let $V$ be a vector space over $\F$.  Let $S\subseteq V$.  Then $\Aff_{\mathbb{F}}(S)=\Span_{\mathbb{F}}(S)$ if and only if $\0\in \Aff_{\mathbb{F}}(S)$.
\end{lemma}
Let $(\rho,V)$ be a representation of a group $G$.  Let $S\subseteq V$ be nonempty.  Then the \textit{barycenter} of $S$ is $\beta(S)=(1/|S|)\sum_{\vv\in S} \vv$.  The \textit{fixed space} of $V$ under the action of $G$ is $$V^G=\{\vv \in V \, | \, \rho(g)\vv=\vv\, \,\forall g\in G\}.$$  Clearly, $V^G$ is a subrepresentation of $V$.  By Theorem~\ref{thm:Maschke}, there exists a subrepresentation orthogonal to $V^G$.  

\begin{lemma}\label{lem:projectionandbarycenterequal}
Let $(\rho,V)$ be a unitary representation of a group $G$.  Let $\PP$ be the orthogonal projection matrix onto $V^G$.  Then $\PP \xx=\beta(G\,\xx)$ for $\xx\in V$.
\end{lemma}
\begin{proof}
Let $\uu_1,\ldots, \uu_k$ be an orthonormal basis for $V^G$.  Then 
$$
\PP\yy=\sum_{i\in[k]}\scalar{\uu_i}{ \yy}\uu_i
$$  
for any $\yy\in V$. Then, for each $ i\in [k]$ and $r \in G$, we have 
$$
\scalar{\uu_i}{\rho(r)\xx}=\scalar{\rho(r)\uu_i}{\rho(r)\xx}=\scalar{\uu_i}{\xx}.
$$
Then 
$$
\scalar{\uu_i}{\beta(G\,\xx)}=\frac{1}{|G\,\xx|}\sum_{\rho(r)\xx\in G\,\xx}\scalar{\uu_i}{\rho(r)\xx}=\frac{1}{|G\,\xx|}\sum_{\rho(r)\xx\in G\,\xx} \scalar{\uu_i}{\xx}=\scalar{\uu_i}{\xx}.
$$
Since $\beta(G\,\xx)\in V^G$, 
\begin{equation*}
    \beta(G\,\xx)=\PP\beta(G\,\xx)=\sum_{i\in [k]} \scalar{\uu_i}{\beta(G\,\xx)} \uu_i=\sum_{i\in [k]} \scalar{\uu_i}{\xx} \uu_i=\PP\xx. \qedhere
\end{equation*}
\end{proof}

\begin{lemma}\label{lem:modlprojection}
Let $\left(R,\Q^{\Z_\ell}\right)$ be the regular $\Q$-representation of $\Z_\ell$.  Then the projection matrix $\PP_\ell$ onto the fixed space of $\Q^{\Z_\ell}$ is $(1/\ell)\J$, where $\J$ is the $\ell\times \ell$ matrix of all ones.
\end{lemma}
\begin{proof}
Notice that the fixed space of $V$ is a direct sum of copies of the trivial representation of $\Z_\ell$.  By Theorem~\ref{thm:projectionoperator}, 
\begin{equation*}
\PP_\ell=\frac{1}{\ell}\sum_{j\in \Z_\ell}\overline{\chi_0(j)}\mathbf{M}_{R(j)}=\frac{1}{\ell}\sum_{j\in \Z_\ell}\mathbf{M}_{R(j)}=\frac{1}{\ell}\J. \qedhere
\end{equation*}
\end{proof}
\section{Main results}\label{sec:mainresult}
We start with the proof of the first of our main results. 

\noindent 
{\bf Proof of Theorem~\ref{thm:main-1}}.
Let $$\mathcal{F}_{\ell}^1=\{\uu \in \{-1,1\}^{\ell}\,|\, \exists \vv \ni (\uu,\vv) \text{ is an LP}\}$$ and $$\mathcal{F}_{\ell}^2=\{\vv \in \{-1,1\}^{\ell}\,|\, \exists \uu \ni (\uu,\vv) \text{ is an LP}\}.$$  By symmetry, $\mathcal{F}_{\ell}^1=\mathcal{F}_{\ell}^2$.  Since $\mathcal{F}_{\ell}\neq \emptyset$, there exists  $(\uu^\top,\vv^\top)^\top\in \mathcal{F}_{\ell}$.  Then $\uu\in \mathcal{F}_{\ell}^1$ and $\vv\in \mathcal{F}_{\ell}^2$.  Let $\pp=\beta(\Z_\ell\, \uu)=-(1/ \ell) \1$.  Then $\pp\in \Conv(\mathcal{F}_{\ell}^1)\subseteq \Aff(\mathcal{F}_{\ell}^1)$.    To reach the desired conclusion, we observe the following facts:
\begin{itemize}
    \item[(i)] Since $\mathcal{F}_{\ell}^1=\mathcal{F}_{\ell}^2$,  Span$_{\C}(\mathcal{F}_{\ell}^1-\pp)=\Span_\C(\mathcal{F}_{\ell}^2-\pp)$.
    \item[(ii)] Since both Span$_{\C}(\mathcal{F}_{\ell}^1-\pp)$ and Span$_{\C}(\mathcal{F}_{\ell}^2-\pp)$ are $\C$-subrepresentations of $(\Q^{\Z_\ell})_\C$ under the action of $\mathbb{Z}_{\ell}$ orthogonal to $\1$, by Theorem~\ref{thm:Maschke}, both Span$_{\C}(\mathcal{F}_{\ell}^1-\pp)$ and Span$_{\C}(\mathcal{F}_{\ell}^2-\pp)$ must be an orthogonal direct sum of the irreducible $\C$-subrepresentations $V_1,\ldots, V_{\ell-1}$ of $(\Q^{\Z_\ell})_\C$. 
    \item[(iii)]Both Span$_{\C}(\mathcal{F}_{\ell}^1-\pp)$ and Span$_{\C}(\mathcal{F}_{\ell}^2-\pp)$ cannot be orthogonal to an irreducible $V_k$ for some $k \in [\ell-1]$.  For this would imply that the DFT of $\uu$ must satisfy
$$\mu_k(\uu)=\mu_k(\uu-\pp)=0,$$
similarly, $\mu_k(\vv)=0$, contradicting Theorem~\ref{thm:PSD}.
\end{itemize}
By (i), (ii), and (iii), 
$$
\Span_\C(\mathcal{F}_{\ell}^1-\pp)=V_1\obot \dots \obot V_{\ell-1} \text{ and } \Span_\C(\mathcal{F}_{\ell}^2-\pp)=V_1\obot \dots \obot V_{\ell-1}.
$$
Let $\ff=\beta\left((\Z_\ell \times \Z_\ell) \left(\uu^\top,\vv^\top\right)^\top\right)=-(1/\ell) \left(\1^\top,\1^\top\right)^\top$.  Then $\ff\in \Conv(\mathcal{F}_{\ell})\subseteq \Aff(\mathcal{F}_{\ell})$.  It is evident that 
$$
\Span_\C\left(\mathcal{F}_{\ell}-\ff\right)=\Span_\C\left(\mathcal{F}_{\ell}^1-\pp\right)\oplus \Span_\C\left(\mathcal{F}_{\ell}^2-\pp\right).
$$
Then by Theorem~\ref{thm:irreducibleQ}, $\Span_\Q(\mathcal{F}_{\ell}^i-\pp)=\displaystyle\bigobot_{\substack{d\,|\, \ell\\ d\neq \ell}}\Col_\Q(\PP_d)$ for $i=1,2$, and by fact (i),
$$
\Span_\Q\left(\mathcal{F}_{\ell}-\ff\right)=\left(\bigobot_{\substack{d\,|\, \ell\\ d\neq \ell}}\Col_\Q(\PP_d)\right)\oplus \left(\bigobot_{\substack{d\,|\, \ell\\ d\neq \ell}}\Col_\Q(\PP_{d})\right).
$$
Hence,
\begin{equation*}
\dim(\Conv(\mathcal{F}_{\ell}))=\dim(\Span(\mathcal{F}_{\ell}-\ff))=\dim_\Q(\Span_\Q(\mathcal{F}_{\ell}-\ff))=2(\ell-1)=2\ell-2. \hfill \qed
\end{equation*}

The following theorem is needed for our next main result.

\begin{theorem}\label{thm:dim}
 Let $\uu\in \Q^\ell$ satisfy $\Span_\Q(\uu)\neq \Span_\Q(\1) $ and let $\yy=\uu-(\1^{\top}\uu/\ell)\1$.  Then there exists $T\subseteq \{ d\in [\ell]: d \mid \ell \text{ and } d\neq \ell  \}$ such that
\begin{equation}\label{eqn:TIrrecucibleSpan}
\Span_\Q(\Z_\ell\,\yy)=\bigobot_{\substack{d\in T}}\Col_{\Q}(\PP_d),
\end{equation}
where $\PP_d$ is defined in Theorem~\ref{thm:irreducibleQ}.
Moreover, 
\begin{align*}
\dim(\Conv(\Z_\ell\, \uu))=\dim(\Aff(\Z_\ell\, \uu))= \dim_{\Q}(\Span_{\Q}(\Z_\ell\, \yy))= {\rm rank}(\CC_{\yy})= \sum_{\substack{d\in T}}\phi\left(\frac{\ell}{d}\right),
\end{align*}
where $\CC_{\yy}$ is the $\ell \times \ell$ circulant matrix whose first row is $\yy$. 
\end{theorem}
\begin{proof}
Since $\Span_\Q(\Z_\ell\, \yy)$ is a $\Q$-subrepresentation of $\Q^{\Z_\ell}$ under the action of $\mathbb{Z}_{\ell}$, $\Span_\Q(\Z_\ell\, \yy)$ is an orthogonal direct sum of irreducible $\Q$-subrepresentations of $\Q^{\Z_\ell}$
by Theorem~\ref{thm:Maschke}. 
By Theorem~\ref{thm:irreducibleQ}, the irreducible $\Q$-subrepresentations of $\Q^{\Z_\ell}$ are $\Col_{\Q}(\PP_d)$ for each divisor $d$ of $\ell$.  Therefore, there exists a set $T$ of divisors of $\ell$ for which equation~(\ref{eqn:TIrrecucibleSpan}) holds.  We first show that $\Span_\Q(\Z_\ell\,\yy)$ is orthogonal only to $\Col_{\Q}(\PP_\ell)$, where $\PP_\ell$ is the projection onto the fixed space of $\Q^{\Z_\ell}$ under the action of $\mathbb{Z}_{\ell}$.  Since $\Col_{\Q}(\PP_\ell)=\Span_\Q(\1)$, $\scalar{\yy}{\1}=0$, and the representation $\Q^{\Z_\ell}$ is unitary, we have $\Span_\Q(\Z_\ell\, \yy)$ is orthogonal to $\Col_{\Q}(\PP_\ell)$.

Let $\zz=\uu-\PP_\ell\uu$.  Then, by Lemma~\ref{lem:modlprojection}, $\zz=\yy=\uu-(\1^{\top}\uu/\ell)\1$.  Since $\PP_\ell\yy=\PP_\ell\uu-\PP_\ell^2\uu=\0$ and $\beta(\Z_\ell\, \yy)=\PP_\ell\yy$ by Lemma~\ref{lem:projectionandbarycenterequal}, $\0=\beta(\Z_\ell\, \yy)\in \Conv_\Q(\Z_\ell\, \yy)$ as $\beta(\Z_\ell\, \yy)$ is a convex combination of points of $\Z_\ell\, \yy$.  Then by Lemma~\ref{lem:zeroinaffine}, we have $\Span_\Q(\Z_\ell\,\yy)=\Aff_\Q(\Z_\ell\, \yy)$.  Observe that
\begin{align*}
\dim_\Q(\Aff_\Q(\Z_\ell\, \uu))&=\dim_\Q\left(\Aff_\Q(\Z_\ell\, \uu)-  \frac{\1^{\top}\uu}{\ell}\1\right)\\
&=\dim_\Q(\Aff_\Q(\Z_\ell\, \yy))=\dim_\Q(\Span_\Q(\Z_\ell\, \yy)).
\end{align*}
Then, 
\begin{equation*}
\dim(\Aff(\Z_\ell\, \uu))=\dim_{\Q}(\Aff_{\Q}(\Z_\ell\, \uu))= \dim_{\Q}(\Span_{\Q}(\Z_\ell\, \yy))=\sum_{\substack{d\in T}}\phi\left(\frac{\ell}{d}\right). \qedhere
\end{equation*}
\end{proof}
We now prove Theorem~\ref{thm:main0}.  
\noindent 
{\bf Proof of Theorem~\ref{thm:main0}}. 

Since $\mathbb{Z}_{\ell} < \Ztimes$, it suffices to prove the result for $G=\mathbb{Z}_{\ell}$.
Let $X_1=\mathbb{Z}_{\ell}\, \uu$, $X_2=\mathbb{Z}_{\ell} \, \vv$. 
%
%
%
Since 
$$\dim(\Conv(X_1))=\dim(\Aff(X_1))\quad \text{and}\quad \dim(\Conv(X_2))=\dim(\Aff(X_2)), $$ by Theorem~\ref{thm:dim}, it suffices to prove that $U_1\cap U_2=\emptyset$. For the sake of contradiction, let $d'\in U_1\cap U_2$ be such that $d'\neq \ell$. This implies that for each $i=1,2$, Span$_{\C}(X_i+(1/\ell)\1)$ is orthogonal to $\Col_{\C}(\PP_{d'})$.  Also, $\Col_{\C}(\PP_{d'})\subset (\mathbb{Q}^{\mathbb{Z}_{\ell}})_{\C}$ and $\Col_{\C}(\PP_{d'})$ is invariant under the action of $\mathbb{Z}_{\ell}$.
Then by Theorems~\ref{thm:Maschke}~and~\ref{thm:Fourier}, there exists $U'\subset\mathbb{Z}_{\ell}-\{0\}$ such that 
  $\Col_{\C}(\PP_{d'})=\bigobot_{i\in U'}V_i$.
  This implies that for each $i=1,2$, Span$_{\C}(X_i+(1/\ell)\1)$ is orthogonal to an irreducible $V_k$ for some $k \in [\ell-1]$.
Then
$$\mu_k(\uu)=\mu_k\left(\uu+\frac{1}{\ell}\1\right)=0,$$
similarly, $\mu_k(\vv)=0$, contradicting Theorem~\ref{thm:PSD}. 
\qed 



 


\begin{lemma}\label{lem:m-2n-2}
Let $m,n\in \Z$.  If $m-2\geq 2,n-2\geq 1$ or $m-2\geq 1,n-2\geq 2$, then $n-2  \geq 2(n-1)/m.$
\end{lemma}
\begin{proof}
If $m-2\geq 2,n-2\geq 1$ or $m-2\geq 1,n-2\geq 2$, then $(m-2)(n-2) \geq 2$.
 Adding $2(n-2)$ to both sides of $(m-2)(n-2) \geq 2$ yields $m(n-2)\geq 2(n-1)$. Hence  $n-2\geq 2(n-1)/m$.
\end{proof}
\begin{lemma}\label{lem:eucalg}
Let $p,q$ be distinct odd primes.  Then the quotient of $2(q-1)(p-1)$ by division of $p$ is at least $q$. Moreover, for  $2(q-1)(p-1)=sp+r$, where $s=2(q-1)-k$, $r=kp-2(q-1)$, and $k$ is the smallest integer such that $kp-2(q-1)\geq 0$, then $s$ and $r$ are the quotient and remainder of $2(q-1)(p-1)$ upon division by $p$.
\end{lemma}
\begin{proof}
 Since for any integer $k$, $2(q-1)(p-1)=(2(q-1)-k)p+(kp-2(q-1))$ we may choose the smallest such $k$ such that $kp-2(q-1)\geq 0$.  Then $k=\big\lceil 2(q-1)/p \big\rceil$, where $\lceil n \rceil$ is the ceiling of $n$.  Let  $2(q-1)(p-1)=sp+r$, where $s=2(q-1)-k$ and $r=kp-2(q-1)$.  Now, as $p,q$ are distinct odd primes, WLOG suppose that  $p\geq 3$, then $q\geq 5\geq 4$.  Since $p-2\geq 1$ and $q-2\geq 2$, by Lemma~\ref{lem:m-2n-2},
 $q-2\geq 2(q-1)/p$.  Therefore, $q-2\geq k$.  Then $s-q=2(q-1)-k-q=q-2-k\geq k-k=0$.  Finally, since $k-1<2(q-1)/p$,  $r=pk-2(q-1)<p$.  
 \end{proof}

\begin{lemma}\label{lem:PSDphi}
Let $\ell=p^\alpha q^\beta$, where $p,q$ are distinct odd primes, and $\alpha,\beta\in \Z_{\geq 1}$.  Then $\phi(\ell)> (\ell-1)/2$.
\end{lemma}
\begin{proof}
By Lemma~\ref{lem:eucalg}, $2(q-1)(p-1)=sp+r$ where $s\geq q$.  Then
\begin{align*}
2\phi(\ell) &=2(q-1)(p-1)p^{\alpha-1}q^{\beta-1}\\
            &=(sp+r)p^{\alpha-1}q^{\beta-1}\\
            &\geq p^\alpha q^\beta + rp^{\alpha-1}q^{\beta-1}\\
            &> p^\alpha q^\beta-1\\
            &=\ell-1,
\end{align*}
and the result follows.
\end{proof}
\begin{lemma}\label{lem:PSDNorm}
Let $\uu\in \Q^{\ell}$,  then
$$
||\mm(\uu)||^2=\ell||\uu||^2.
$$
\end{lemma}
\begin{proof}
 Let $\U$ be the matrix whose columns are the discrete Fourier basis vectors $\vv_0, \ldots, \vv_{\ell-1}$.
Then, since $\mm=\U^\top \uu$,
$$
||\mm(\uu)||^2=\sum_{k\in \Z_\ell}|\mu_k(\uu)|^2=\mm^*\mm=\uu^\top \U^* \U \uu=\uu^\top \ell \mathbf{I}_{\mathbb{Q}^\ell} \uu=\ell ||\uu||^2,
$$
where we used the fact that $\U^*\U=\ell  \mathbf{I}_{\mathbb{Q}^\ell}$.
\end{proof}

\begin{corollary}\label{cor:PSD}
Let $\uu\in \Q^{\Z_\ell}$ satisfying $||\uu||^2=\ell$ and $\scalar{\1}{\uu}=-1$. Then
$$
\sum_{k=1}^{\ell-1}|\mu_k(\uu)|^2=(\ell+1)(\ell-1).
$$
\end{corollary}
\begin{proof}
By Lemma~\ref{lem:PSDNorm},
\begin{equation*}
\sum_{k=1}^{\ell-1}|\mu_k(\uu)|^2=||\mm(\uu)||^2-|\mu_0(\uu)|^2=\ell^2-1=(\ell+1)(\ell-1). \qedhere
\end{equation*}
\end{proof}

For each $\ell,n\in \mathbb{Z}_{\geq 1}$, the Ramanujan's sum~\cite{apostol1998introduction} is defined by $c_{\ell}(n)=\sum_{ \substack{0<k\leq \ell \\ (k,\ell)=1}} e^{2 \pi i k n /\ell}$.  It is well-known that  $\forall \ell,n\in \mathbb{Z}_{\geq 1}$ $$c_{\ell}(n)=\mu\left(\frac{\ell}{(\ell,n)}\right)
\frac{\phi(\ell)}{\phi\left(\frac{\ell}{(\ell,n)}\right)}\in \mathbb{Z},$$ where
\begin{eqnarray*}
\mu(n)=\left\{
\begin{array}{cl}
\text{$1$} & \quad \text{if $n=1$,}\\
\text{$(-1)^r$} &\quad  \text{if $n=p_1\ldots p_r $ for distinct primes $p_1,\ldots,p_r$,}\\
\text{$0$} &\quad  \text{if $n$ is divisible by some prime square}
\end{array}\right.
\end{eqnarray*}
is the  M\"obius function.
\begin{lemma}\label{lem:Ramanujan}
Let $\ell>1$ and $\uu\in \Q^{\ell}$. 
If $d$ is a divisor of $\ell$,  then
\begin{equation*}
\sum_{ \substack{0<k\leq\ell \\ (k,\ell)=d}}|\mu_k(\uu)|^2=\sum_{s\in \Z_\ell}c_{\frac{\ell}{d}}(s)P_{\uu}(s).
\end{equation*}
\end{lemma}
\begin{proof}
First, 
$$
\sum_{\substack{0<k\leq\ell \\ (k,\ell)=d}}e^{\frac{2 \pi i k(h-j)}{\ell}}=
\sum_{\substack{0<r\leq \frac{\ell}{d} \\ (r,\frac{\ell}{d})=1}}e^{\frac{2 \pi i r(h-j)}{\frac{\ell}{d}}}=
c_{\frac{\ell}{d}}(h-j).
$$
Let $\vv_0, \vv_{1},\ldots, \vv_{\ell-1}$ be the discrete Fourier basis vectors.  Note that $\vv_k\vv_k^*=[a^k_{hj}]$, where $a^k_{hj}=e^{2 \pi i k(h-j)/\ell}$.  Then
$$
\sum_{\substack{0<k\leq\ell \\ (k,\ell)=d}}\vv_k\vv_k^*=[b^d_{hj}],
$$
where $b^d_{hj}=c_{\ell /d}(h-j)$.  Since $\mu_k(\uu)=\uu^\top \vv_k$,
\begin{equation*}
|\mu_k(\uu)|^2=\mu_k(\uu) \overline{\mu_k(\uu)}=(\uu^\top \vv_k) (\vv_k^* \uu)=\uu^\top (\vv_k \vv_k^*)\uu.
\end{equation*}
Then 
\begin{align*}\label{eqn:RamanujanSS}
\begin{split}
\sum_{\substack{0<k\leq\ell \\ (k,\ell)=d}}|\mu_k(\uu)|^2 
&=                                                                
\uu^\top\left(\sum_{\substack{0<k\leq\ell \\ (k,\ell)=d}}\vv_k\vv_k^*\right)\uu\\ 
&=
\sum_{h,j\in \Z_\ell}b^d_{hj} u_h u_j\\
&=
\sum_{h,j\in \Z_\ell}c_{\frac{\ell}{d}}(h-j) u_h u_j\\
&=
\sum_{s,j\in \Z_\ell}c_{\frac{\ell}{d}}(s)u_j u_{j+s}\\
&=
\sum_{s\in \Z_\ell}c_{\frac{\ell}{d}}(s)\sum_{j\in \Z_\ell}u_j u_{j+s}\\
&=
\sum_{s\in \Z_\ell}c_{\frac{\ell}{d}}(s)P_{\uu}(s). \qedhere
\end{split}
\end{align*}
\end{proof}

For a vector $\uu\in \Q^{\ell}$ and a divisor $d$ of $\ell$,  let $$\uu^{\frac{\ell}{d}}(j)=\sum_{i\in \mathbb {Z}_{d}}\uu\left(\frac{\ell}{d}i+j\right), \quad \forall j \in \mathbb{Z}_{\frac{\ell}{d}}.$$ Then $\uu^{\ell/d}\in \Q^{\ell/d}$ is called the {\em $d$-compression} of $\uu$. 
The following lemma is used to prove one of our main results.
\begin{lemma}\label{lem:PSDdcompression}
Let $\ell>1$ and $\uu\in \Q^{\ell}$.  If $d$ is a divisor of $\ell$,  then
\begin{equation*}
\sum_{ \substack{0<k\leq\ell \\ (k,\ell)=d}}|\mu_k(\uu)|^2=\sum_{\substack{0<k\leq\frac{\ell}{d} \\ (k,\frac{\ell}{d})=1}}|\mu_k(\uu^\frac{\ell}{d})|^2. 
\end{equation*}
\end{lemma}
\begin{proof}
 By Lemma~\ref{lem:Ramanujan}, 
 \begin{align*}
\begin{split}
\sum_{\substack{0<k\leq\ell \\ (k,\ell)=d}}|\mu_k(\uu)|^2 
&=         
\sum_{s\in \Z_\ell}c_{\frac{\ell}{d}}(s)P_{\uu}(s)\\
&=
\sum_{j\in \Z_{\frac{\ell}{d}}}\sum_{i\in \Z_d}c_{\frac{\ell}{d}}\left(\frac{\ell}{d}i+j\right)P_{\uu}\left(\frac{\ell}{d}i+j\right)\\
&=
\sum_{j\in \Z_{\frac{\ell}{d}}}c_{\frac{\ell}{d}}\left(\frac{\ell}{d}+j\right)\sum_{i\in \Z_d}P_{\uu}\left(\frac{\ell}{d}i+j\right)\\
&=
\sum_{s\in \Z_{\frac{\ell}{d}}}c_{\frac{\ell}{d}}(s)P_{\uu^\frac{\ell}{d}}(s)\\
&=
\sum_{\substack{0<k\leq\frac{\ell}{d} \\ (k,\frac{\ell}{d})=1}}|\mu_k(\uu^\frac{\ell}{d})|^2. \qedhere
\end{split}
\end{align*}
\end{proof}

The following theorem follows from Theorem 3 in~\cite{Dragomir2013}.
\begin{theorem}\label{thm:Dragomir}
 Let $(\uu,\vv)$ be an LP of length $\ell=dm$. Then the corresponding $m=\ell/d$-compressed sequences $(\uu^d,\vv^d)$ satisfy
  \begin{align*}
P_{\uu^d}(0)+P_{\vv^d}(0)&=2\ell-2\left(\frac{\ell}{d}-1\right)=2\ell+2-\frac{2\ell}{d},\label{eqn:PAFcompconstraint1}\\
P_{\uu^d}(j)+P_{\vv^d}(j)&=-\frac{2\ell}{d}, \quad \forall j\in \Z_d-\{0\}.
\end{align*}
\end{theorem}
We now prove another main result.
\begin{theorem}\label{thm:goal}
Let  $(\uu,\vv)$ be an LP of length $\ell$ such that 
$||\uu^d||^2=||\vv^d||^2 $ for some divisor $d$ of $\ell$ such that $2\phi(d)\geq d-1$. Then, $|\mu_{k}(\vv)|^2>0$ and $|\mu_{k}(\uu)|^2>0$ $\forall k \ni (k,\ell)=\ell/d$.
\end{theorem}
\begin{proof}
By symmetry, it suffices to prove the result for $\vv$.
For the sake of contradiction, let 
$|\mu_{k'}(\vv)|^2=0$ for some $k' \ni (k',\ell)=\ell/d$ for some divisor $d$ of $\ell$. Then $|\mu_{k}(\vv)|^2=0$ $\forall k$ $\ni (k,\ell)=\ell/d$.  By Theorem~\ref{thm:PSD}, $|\mu_k(\uu)|^2=2\ell+2$ $\forall k \ni (k,\ell)=\ell/d$.  Consequently,
\begin{equation}\label{eqn:phid}
\sum_{ \substack{0<k\leq\ell \\ (k,\ell)=\frac{\ell}{d}}}|\mu_k(\uu)|^2=
(2\ell+2)\phi(d).\end{equation}
Then, by Lemma~\ref{lem:PSDdcompression},
 \begin{align}\label{eqn:summukloverd}
\sum_{ \substack{0<k\leq\ell \\ (k,\ell)=\frac{\ell}{d}}}|\mu_k(\uu)|^2=\sum_{\substack{0<k\leq d \\ (k,d)=1}}|\mu_k(\uu^d)|^2=(2\ell+2)\phi(d). 
\end{align}
By Theorem~\ref{thm:Dragomir}, 
$$||\uu^d||^2=||\vv^d||^2=\ell+1-\frac{\ell}{d}. $$
Then, by Lemma~\ref{lem:PSDNorm},  
\begin{equation}\label{eqn:summuk}
\sum_{k=1}^d|\mu_k(\uu^d)|^2=d||\uu^d||^2-|\mu_0(\uu^d)|^2=d||\uu^d||^2-1.
\end{equation}
Hence, by equations~(\ref{eqn:phid}),~(\ref{eqn:summukloverd}), and~(\ref{eqn:summuk})
\begin{equation*}\label{eqn:maxmuksum}
\sum_{ \substack{0<k\leq d \\ (k,d)=1}}|\mu_k(\uu^d)|^2=(2\ell+2)\phi(d) <  \sum_{k=1}^d|\mu_k(\uu^d)|^2=(\ell+1)(d-1) 
\end{equation*}
which implies
$$2\phi(d)<(d-1),  $$
a contradiction.
\end{proof}
The following corollary follows from Lemma~\ref{lem:PSDphi} and  Theorem~\ref{thm:goal}. 
\begin{corollary}\label{cor:T1T2}
Let $\left(\uu,\vv\right)$ be an LP of length $\ell=p^{\alpha}q^{\beta}$, where $p,q$ are distinct odd primes, and $\alpha,\beta\in \Z_{\geq 1}$. Let $$T=\left\{d\in [\ell] : d \neq 1, d\mid\ell,  \text{and } ||\uu^d||^2\neq||\vv^d||^2\, \right\}.$$ Then
\begin{align*}\label{eqn:dimleqcon}
\begin{split}
\dim(\Conv(\Z_\ell\,\uu))&= \ell-1-\sum_{d \in T_1}\phi\left(d\right),\\ \dim(\Conv(\Z_\ell\,\vv))&= \ell-1-\sum_{d \in T_2}\phi\left(d\right)
\end{split}
\end{align*}
for some $T_1 \subseteq T $ and $T_2 \subseteq T $.
Equivalently 
\begin{align*}
\begin{split}
\dim(\Conv(\Z_\ell\,\uu))&\geq \sum_{d \in T'}\phi\left(d\right),\\ \dim(\Conv(\Z_\ell\,\vv))&\geq \sum_{d \in T'}\phi\left(d\right),
\end{split}
\end{align*}
where  
\begin{equation}\label{eqn:T'd}
T'=\left \{d\in [\ell] : d \neq 1, d \mid \ell,\,  \text{and } ||\uu^d||^2 = ||\vv^d||^2 \right\}.
\end{equation}
\end{corollary}
\begin{proof}
By symmetry, it suffices to prove the result for $\uu$. 
 Let $\yy=\uu+(1/\ell) \1$.  First, we prove that   $$ \Col_\Q(\PP_{\frac{\ell}{d'}})\subseteq \Span_\Q(\Z_\ell\,\yy)\,\, \forall d' \in T'.$$ 
 The subspace $\Span_\Q(\Z_\ell\,\yy)$ is a  $\mathbb{Q}$-subrepresentation  of $\Q^\ell$,
 and thus either $$\Span_\Q(\Z_\ell\,\yy)\perp  \Col_\Q\left(\PP_{\frac{\ell} {d'}}\right)$$ or $$ \Col_\Q\left(\PP_{\frac{\ell} {d'}}\right)\subseteq \Span_\Q(\Z_\ell\,\yy).$$ If $\Span_\Q(\Z_\ell\,\yy)\perp  \Col_\Q\left(\PP_{\ell/d'}\right)$, then $\Span_\C(\Z_\ell\,\yy)\perp  \Col_\C\left(\PP_{\ell/d'}\right)$, and consequently 
 $\mu_k(\uu)=\mu_k(\yy)=0$
 for some $k \ni (k,\ell)=\ell/d'$. By Lemma~\ref{lem:PSDphi}, this contradicts 
 Theorem~\ref{thm:goal}. Hence, $\Col_\Q\left(\PP_{\ell /d'}\right)\subseteq \Span_\Q\left(\Z_\ell\,\yy\right)$ and
 the result now follows from Theorem~\ref{thm:dim}.
\end{proof}
 Since 
$||\uu^\ell||^2 = ||\vv^\ell||^2=\ell$ for $\uu,\vv\in \{-1,1\}^{\ell}$, the following corollary follows from the proof of Corollary~\ref{cor:T1T2}.
\begin{corollary}\label{cor:P1}
Let $\left(\uu,\vv\right)$ be an LP of length $\ell=p^{\alpha}q^{\beta}$, where $p,q$ are distinct odd primes, and $\alpha,\beta\in \Z_{\geq 1}$. Then  
\begin{itemize}
\item[(i)] 
$\Col_\Q\left(\PP_1\right)\subseteq \Span_\Q(\Z_\ell\,\yy)$ for $\yy=\uu+(1/\ell) \1$ and $\yy=\vv+(1/\ell) \1$,
\item[(ii)]
$\dim(\Conv(\Z_\ell\,\uu)) \geq \phi\left(\ell\right)$ and $\dim(\Conv(\Z_\ell\,\vv)) \geq \phi\left(\ell\right)$.
\end{itemize}
\end{corollary}
Larger $|T'|$ in equation~(\ref{eqn:T'd}) in Corollary~\ref{cor:T1T2} yields larger lower bounds for
$\dim(\Conv(\Z_\ell\,\uu))$ and $\dim(\Conv(\Z_\ell\,\vv))$. There is evidence that 
in general $T' \neq \emptyset$. In fact, 
 based on known LPs, it was recently conjectured in~\cite{KKBT:2021} that for each positive $\ell$ divisible by $5$
$$5 \in \left \{d\in [\ell] :  d \neq 1,\ d \mid \ell,\ ||\uu^d||^2 = ||\vv^d||^2, \text{ and } (\uu,\vv) \text{ is an LP of length $\ell$}\right\}.$$
On the other hand, there are LPs $(\uu, \vv)$ such that
$||\uu^d||^2 \neq ||\vv^d||^2$ for some divisor $d$ of $\ell$, see~\cite{KK:JCD:2021,KKBT:2021}.
The following theorem gives us a general lower bound on
$\dim(\Conv(\Z_{\ell}\,\uu))$ for $\uu \in \{-1,1\}^{\ell}$ such that $\scalar{\1}{\uu}=-1$.
\begin{theorem}\label{thm:ranklowerlimit}
 Let $\ell=p_1^{n_1}\dots p_s^{n_s}$ where $p_1,\dots, p_s$ are distinct odd primes and $n_1,\dots, n_s\in \Z_{\geq 1}$.  Let $\uu \in \{-1,1\}^{\ell}$ satisfy $\scalar{\1}{\uu}=-1$, then $$\dim(\Conv(\Z_{\ell}\,\uu))\geq \sum_{j\in [s]}\sum_{i\in [n_j]}\phi(p_j^i).$$
 \end{theorem}
 \begin{proof}
 Let 
$\yy=\uu+(1/\ell) \1$.  We will first prove that $\Col_\Q(\PP_{d_{j,i}})\subseteq \Span_\Q(\Z_\ell\,\yy)$ for each $d_{j,i}=p_1^{n_1}\ldots p_j^i\ldots p_s^{n_s},\, j=1,\ldots,s , \, i=0,\ldots,n_j-1$.
Assume otherwise, and let $\mu_k(\yy)=0$ for some $k \ni (k,\ell)=d_{j,i}$.  Then $\mu_k(\uu)=\mu_k(\yy)=0$, and by equation~(\ref{eqn:lambdatwoforms}) and
Theorem~\ref{thm:Lam}, $p_j\,|\, (\ell+1)/2$, a contradiction.  Therefore, $\Col_\Q(\PP_{d_{j,i}})\subseteq \Span_\Q(\Z_\ell\,\yy)$, and by Theorem~\ref{thm:dim},
\begin{align*}
\dim(\Conv(\Z_{\ell}\,\uu))&=\dim(\Aff(\Z_\ell\,\uu))=\dim(\Span(\Z_\ell\,\yy))\\
&=\dim_\Q(\Span_\Q(\Z_\ell\,\yy))\geq \sum_{j\in [s]}\sum_{i\in [n_j]}\phi(p_j^i). \qedhere
\end{align*}
 \end{proof}
The following corollary follows by combining Corollary~\ref{cor:P1} and the proof of  Theorem~\ref{thm:ranklowerlimit}.
\begin{corollary}\label{cor:Krisfollows}
Let $\left(\uu,\vv\right)$ be an LP of length $\ell=p^{\alpha}q^{\beta}$, where $p,q$ are distinct odd primes, and $\alpha,\beta\in \Z_{\geq 1}$. Then  
\begin{align*}
\begin{split}
\dim(\Conv(\Z_{\ell}\,\uu))\geq \sum_{j\in [\alpha]}\phi(p^j)+\sum_{i\in [\beta]}\phi(q^i)+\phi(\ell),\\ \dim(\Conv(\Z_{\ell}\,\vv))\geq \sum_{j\in [\alpha]}\phi(p^j)+\sum_{i\in [\beta]}\phi(q^i)+\phi(\ell).
\end{split}
\end{align*}
\end{corollary}

A different set of lower bounds on $\dim(\Conv(\Z_{\ell}\uu))$ can be obtained by using the results  in Ingleton~\cite{ingleton1956rank}. Next, we derive such lower bounds when $\ell=pq^{\beta}$ for some distinct odd primes $p,q$ and positive integer $\beta$. However, first we need the concept of non-recurrent matrices.

Let $\CC_\uu$ be the circulant matrix of $\uu\in \R^\ell$.  Then $\CC_\uu$ is \textit{non-recurrent} if $\ell$ is the only divisor $d$ of $\ell$ such that $u_i=u_j$ whenever $i\equiv j \pmod d$.  The following lemma shows that $\CC_\uu$ non-recurrent if there is  a $\vv$ such that $(\uu,\vv)$ is an LP. 
\begin{lemma}
Let $\uu\in \{-1,1\}^\ell$ and $\scalar{\1}{\uu}=\mp1$.  Then $\CC_\uu$ is non-recurrent.
\end{lemma}
\begin{proof}
 First note that the constraint $\scalar{\1}{\uu}=\mp1$ implies that there are $(\ell\pm1)/2$ $-1$'s and $(\ell\mp 1)/2$ $1$'s.  Suppose for contradiction that $\CC_\uu$ is $s$-recurrent.  Then $s\mid \ell$ and $u_i=u_j$ whenever $i\equiv j \pmod s$.  Since $u_0,\ldots, u_{s-1}$ each appear $d=\ell/s$ times and $d$ divides neither $(\ell\pm1)/2$ nor $(\ell\mp1)/2$, we get a contradiction.
\end{proof}

The following lemma and the subsequent corollary are needed to derive  lower bounds on the rank of circulant  matrices $\CC_\uu$ for $\uu \in \{-1,1\}^{\ell}$ based on rank lower bounds in~\cite{ingleton1956rank} for  circulant $\CC_\uu$ with $\uu \in \{0,1\}^{\ell}$.
\begin{lemma}\label{lem:all1s}
Let $\CC_\uu$ be the circulant matrix of $\uu\in \R^\ell$, $\J$ be the $\ell\times \ell$ all $1$s
 matrix, and $\lambda \in \R$. Then
 \begin{itemize}
\item [(i)] the discrete Fourier basis $\{\vv_0,\vv_1,\ldots,\vv_{\ell-1}\}$ is an eigenbasis for both $\CC_\uu$ and $\CC_\uu+\lambda \J$,
\item [(ii)] if the eigenvalue of $\CC_\uu$ corresponding to eigenvector $\vv_i$ is $\mu_i$ for $i=0,\ldots,\ell-1$, then the eigenvalue of $\CC_\uu+\lambda \J$ corresponding to eigenvector $\vv_0$ and $\vv_i$ are $\mu_0+\ell \lambda$ and $\mu_i$ for $i=1,\ldots,\ell-1$, respectively, where $\mu_0=\sum_{i=0}^{\ell-1}u_i$.
\end{itemize}
\end{lemma}
\begin{proof}
It is straightforward to check the statements of the lemma.  
  \end{proof}
The following corollary follows from Lemma~\ref{lem:all1s}. 
\begin{corollary}\label{cor:circell}
Let $\uu, \CC_\uu, \J, \lambda,$ and $\mu_0$  be as in Lemma~\ref{lem:all1s}. Then
\begin{itemize}
\item [(i)]
\begin{equation*}
{\rm rank}(\CC_\uu+\lambda \J)=
\begin{cases}
{\rm rank}(\CC_\uu) &\text{if} \quad (\mu_0+\ell \lambda) \mu_0 \neq 0,\\
{\rm rank}(\CC_\uu)+1 &\text{if} \quad \mu_0=0 
\text{ and } \lambda \neq 0,  \\
{\rm rank}(\CC_\uu)-1 &\text{if} \quad \mu_0+\ell \lambda =0 \text{ and } \lambda \neq 0,
\end{cases}  
\end{equation*}
\item [(ii)]
for odd $\ell$ and $\uu\in \{0,1\}^{\ell}$,
$${\rm rank}(\CC_{2\uu-\1})= {\rm rank}(\CC_{\uu}).$$
\end{itemize}
\end{corollary}
\begin{proof}
By Lemma~\ref{lem:all1s}, (i) follows immediately. 
  Since $\CC_{2\uu-\1}=2\CC_{\uu}-\J$, and for odd $\ell$,
\begin{equation*}
    {\rm rank}(2\CC_{\uu}-\J)={\rm rank}(2\CC_{\uu})={\rm rank}(\CC_{\uu}),
\end{equation*}
(ii) follows.
\end{proof}
For $\ell=p^\alpha p_1^{\alpha_1} \dots p_{m-1}^{\alpha_{m-1}}$ where $p,p_1,\dots,p_{m-1}$ are distinct primes and $\alpha,\alpha_1,\dots,$ $ \alpha_{m-1}\in \Z_{\geq 1}$, let $$\tau(\ell,p)=1+\epsilon(m)\epsilon(\alpha)\phi(p)+\phi(p^\alpha)+\sum_{i\in [m-1]}\phi(pp^{\alpha_i}_{i}),$$ where $\epsilon(1)=0$ and $\epsilon(k)=1$ for $k\in \Z_{\geq 1}$.
Now, the following lemma follows from the corresponding result for circulant $\ell \times \ell$ non-recurrent circulant matrix with entries from $\{0,1\}$ in Ingleton~\cite{ingleton1956rank} and Corollary~\ref{cor:circell}.

\begin{lemma}\label{lem:Ingleton}
Let $\ell=pq^\beta$, where $p,q$ are distinct odd  primes, and $\beta\in \Z_{\geq 1}$.  Let $\uu\in\{-1,1\}^\ell$ satisfy $\scalar{\1}{\uu}=-1$.  Then ${\rm rank}(\CC_\uu)\geq \min\{\tau(\ell,p),\tau(\ell,q)\}$. 
\end{lemma}

If $\beta\geq 2$, then the following lemma implies that the rank is at least $\tau(\ell,q)$.

\begin{lemma}\label{lem:pqbeta}
Let $\ell=pq^\beta$ for distinct odd primes $p,q$, and $\beta \in \Z_{\geq 2}$.  Then $\tau(\ell,q)<\tau(\ell,p)$.
\end{lemma}
\begin{proof}
Now 
$$
\tau(\ell,q)=1+\epsilon(2)\epsilon(\beta)\phi(q)+\phi(q^\beta)+\phi(pq)=1+\phi(q)+\phi(q^\beta)+\phi(pq),$$ and $$\tau(\ell,p)=1+\phi(p)+\phi(pq^\beta).$$
Then
\begin{align*}
    \tau(\ell,p)-\tau(\ell,q) &=\phi(q)(q^{\beta-1}(\phi(p)-1)-(\phi(p)+1))+\phi(p)\\
    &=\phi(q)(q^{\beta-1}(p-2)-p)+\phi(p).
\end{align*}

We consider two cases: $p=3$ and $p\geq 5$.  If $p=3$, then 
$$
\tau(\ell,p)-\tau(\ell,q) =                 \phi(q)(q^{\beta-1}-3)+\phi(3)\geq \phi(q)(5^{\beta-1}-3)+\phi(3)>0.
$$
Suppose that $p\geq 5$.  Then since $q^{\beta-1}-1\geq 1$ and $p-3\geq 1$, $(q^{\beta-1}-1)(p-3)\geq 1$.  Then $q^{\beta-1}(p-2)-p\geq q^{\beta-1}-2\geq 3^{\beta-1}-2>0.$
Implying
\begin{equation*}
\tau(\ell,p)-\tau(\ell,q)=\phi(q)(q^{\beta-1}(p-2)-p))+\phi(p)>0. \qedhere
\end{equation*}
\end{proof}

Now, the following corollary follows from Theorem~\ref{thm:ranklowerlimit}, Corollary~\ref{cor:circell},  
and Lemmas~\ref{lem:Ingleton}~and~\ref{lem:pqbeta}.
\begin{corollary}\label{cor:pqbeta}
Let $\ell=pq^\beta$ for distinct odd primes $p, q$, and $\beta \in \Z_{\geq 2}$.  Let $\CC_{\uu}$ with $\uu\in \{-1,1\}^{\ell}$ be the  circulant matrix of $\uu$ such that $\CC_{\uu}$ is non-recurrent.  Then  $${\rm rank}(\CC_{\uu})\geq L(pq^{\beta})=\max\left\{\phi(pq)+\phi(q)+\phi(q^\beta)+1,\,\, \phi(p)+\sum\limits_{i\in[\beta]}\phi(q^i)+1\right\}.$$
\end{corollary}
If $\phi(pq)+\phi(q)+\phi(q^\beta)> \phi(p)+\sum\limits_{i\in[\beta]}\phi(q^i)$, we may improve the lower bound in Corollary~\ref{cor:pqbeta}.  We have the following theorem.
\begin{theorem}\label{thm:improvedL}
Let $\ell=pq^\beta$ for distinct odd primes $p, q$, and $\beta \in \Z_{\geq 2}$.   Let $\CC_{\uu}$ with $\uu\in \{-1,1\}^{\ell}$ be the  circulant matrix of $\uu$ such that $\CC_{\uu}$ is non-recurrent. Then 
\begin{itemize}
\item[(i)]
${\rm rank}(\CC_\uu)$ is at least
\begin{align*}
\phi(p)+\phi(pq)+\sum\limits_{i\in[\beta]}\phi(q^i)+1 &\quad \text{if } \phi(pq)+\phi(q)+\phi(q^\beta)> \phi(p)+\sum\limits_{i\in[\beta]}\phi(q^i),\\
\phi(p)+\sum\limits_{i\in[\beta]}\phi(q^i)+1&\quad \text{otherwise,}
\end{align*}
\item[(ii)]
$\dim(\Conv(\Z_\ell\uu))$ is at least
\begin{align*}
\phi(p)+\phi(pq)+\sum\limits_{i\in[\beta]}\phi(q^i) &\quad \text{if } \phi(pq)+\phi(q)+\phi(q^\beta)> \phi(p)+\sum\limits_{i\in[\beta]}\phi(q^i),\\
\phi(p)+\sum\limits_{i\in[\beta]}\phi(q^i) &\quad \text{otherwise.}
\end{align*}
\end{itemize}
\end{theorem}
\begin{proof}
  We first show that (ii) follows from (i). Let $\yy =\uu-(\1^{\top}\uu/\ell)\1 $. Then by Theorem~\ref{thm:dim},
  \begin{align*}
\dim(\Conv(\Z_{\ell}\,\uu))&=\dim(\Aff(\Z_\ell\,\uu))=
\dim(\Span(\Z_\ell\,\yy))=\text{rank}(\CC_\yy).
\end{align*}
Since $ \CC_{\yy} =\CC_{\uu}-((\1^{\top}\uu)/\ell)\J$, by part (i) of Corollary~\ref{cor:circell},  
  $$\dim(\Conv(\Z_{\ell}\,\uu))={\rm rank}(\CC_\yy)={\rm rank}(\CC_\uu)-1.$$
 To prove (i), suppose that $\phi(pq)+\phi(q)+\phi(q^\beta)> \phi(p)+\sum\limits_{i\in[\beta]}\phi(q^i)$.  By Corollary~\ref{cor:circell} and Theorem~\ref{thm:dim}, 
 $${\rm rank}(\CC_\uu)={\rm rank}(\CC_\yy)+1=1+\phi(p)+\sum\limits_{i\in[\beta]}\phi(q^i)+\sum\limits_{i\in[\beta]}a_i\phi(pq^i),$$ where $a_i\in \{0,1\},\, i\in [\beta]$.  If $a_i=0$ $ \forall i\in [\beta]$, then $${\rm rank}(\CC_\uu)=1+\phi(p)+\sum\limits_{i\in[\beta]}\phi(q^i)< 1+\phi(pq)+\phi(q)+\phi(q^\beta),$$ contradicting Corollary~\ref{cor:pqbeta}. Hence, at least one of the $a_i$s must be equal to $1$.  Since $\phi(pq^{i_1})>\phi(pq^{i_2})$ if and only if $i_1>i_2$,
 $a_1=1$, $a_i=0$ for $i=2,\ldots,\beta$ results in the smallest value of ${\rm rank}(\CC_\uu)$ that we can not rule out.
\end{proof}
The following corollary follows immediately from Theorems~\ref{thm:dim}~and~\ref{thm:improvedL}.
\begin{corollary}\label{cor:pq2}
Let $\ell=pq^2$, where $p,q$ are distinct odd primes.   Let $\CC_{\uu}$ with $\uu\in \{-1,1\}^{\ell}$ be the  circulant matrix of $\uu$ such that $\CC_{\uu}$ is non-recurrent.  Then ${\rm rank}(\CC_\uu)$ is at least $$\phi(p)+\phi(pq)+\phi(q)+\phi(q^2)+1=\ell-\phi(\ell).$$
Consequently, either $$\dim(\Conv(\Z_\ell\,\uu))= \ell-1-\phi(\ell),$$ or 
$$\dim(\Conv(\Z_\ell\,\uu))= \ell-1.$$
\end{corollary}

The following corollary follows from Lemmas~\ref{lem:vanishingprimepower},~\ref{lem:vanishingdistinctproductprime}, and~\ref{lem:inner}, and Theorem~\ref{thm:dim}. 

\begin{corollary}\label{cor:dimpn}
Let $p$, $q$ be distinct odd primes and $n \in \mathbb{Z}_{\geq 1}$. Suppose that $\ell=p^n$ or
$\ell=pq$.
 Let $\uu\in \{-1,1\}^\ell$ satisfy $\scalar{\1}{\uu}=-1$ and let $\yy=\uu+(1/\ell)\1$.  Then 
$$
\Span_\Q(\Z_\ell\,\yy)=\bigobot_{\substack{d\,|\,\ell\\ d\neq \ell}}\Col_{\Q}(\PP_d),
$$
where $\PP_d$ is defined in Theorem~\ref{thm:irreducibleQ}.
Moreover, 
\begin{equation}\label{eqn:dimell-1}
\dim(\Conv(\Z_\ell\,\uu))=\dim(\Aff(\Z_\ell\, \uu))=\sum_{\substack{d\,|\,\ell\\ d\neq \ell}}\dim_{\Q}(\Col_{\Q}(\PP_d))=\ell-1.
\end{equation}
\end{corollary}
\begin{proof}
Let $d\neq \ell$.  Since $(\Col_{\Q}(\PP_d))_\C=\Col_{\C}(\PP_d)$ is spanned by the discrete Fourier basis vectors $\vv_k$ for $k\in \Z_\ell \ni (k,\ell)=d$, and $\mu_k(\yy)=\mu_k(\uu)$.
By Lemmas~\ref{lem:vanishingprimepower}~and~\ref{lem:vanishingdistinctproductprime}, $\mu_k(\yy)\neq 0$.  Since this holds $\forall k \ni (k,\ell)=d$, by Lemma~\ref{lem:inner}, 
$\Span_\Q(\Z_\ell\yy)$ is not orthogonal to $\Col_{\Q}(\PP_d)$.  As this holds for each divisor $d\neq \ell$ of $\ell$, by Theorem~\ref{thm:dim}, we must have 
$$
\Span_\Q(\Z_\ell\,\yy)=\bigobot_{\substack{d\,|\,\ell\\ d\neq \ell}}\Col_{\Q}(\PP_d).
$$ 
Equation~(\ref{eqn:dimell-1}) now follows
from Theorem~\ref{thm:dim}.
\end{proof}
The following lemma follows from
Corollaries~\ref{cor:Krisfollows} and \ref{cor:pq2}. 
\begin{lemma}\label{lem:pq2} 
Let $\ell=pq^2$, and $\left(\uu,\vv\right)$ be an LP of length $\ell$, then  $$\dim(\Conv(\Z_\ell\,\uu))=\dim(\Conv(\Z_\ell\,\vv))= \ell-1.$$
\end{lemma}
We now prove Theorem~\ref{thm:main1}.\\
\noindent 
{\bf Proof of Theorem~\ref{thm:main1}}. 
Since $\Z_\ell < (\Ztimes) $, 
it suffices to prove the result for $G=\Z_\ell$.
By Corollary~\ref{cor:dimpn} and Lemma~\ref{lem:pq2},
\begin{align*}
\ell-1&=\dim_{\Q}\left(\Aff_{\Q}\left(\Z_\ell\, \uu\right)\right)=\dim(\Aff(\Z_\ell\, \uu))\\
&=\dim(\Conv(\Z_\ell\, \uu)).
\end{align*}
The result now follows from Theorem~\ref{thm:main0} as $U=\emptyset$ is the only possibility.
\qed

%

Based on our partial results, we end the paper with the following conjecture.
\begin{conjecture}\label{conj:Dursun1}
Let $(\uu,\vv)$ be a LP of length $\ell$. Then  
$$\dim(\Conv(G\,\uu)) =\dim(\Conv(G\,\vv))= \ell-1$$
for $G=\mathbb{Z}_{\ell}$ and  $G=\Ztimes$.
\end{conjecture} 
It suffices to prove Conjecture~\ref{conj:Dursun1} for only one of
 $G=\mathbb{Z}_{\ell}$ or  $G=\Ztimes$. This is because by Theorem~\ref{thm:Vsemirep} 
the decomposition of $\Q^{\Z_{\ell}}$ into irreducible 
$\Q$-subrepresentations is the same under $G=\mathbb{Z}_{\ell}$ or  $G=\Ztimes$. 

\section{Discussion}\label{sec:discussion}
In this paper, we made progress towards proving the conjecture that  for an LP $(\uu,\vv)$ of length $\ell$, 
$$\dim(\Conv(G\,\uu)) =
\dim(\Conv(G\,\vv))= \ell-1$$ for $G=\Z_{\ell}$ and $G=\Ztimes$.
The immediate implication of this conjecture
is that the search space for LPs
cannot be decreased by imposing additional linearly independent linear equality constraints on $\uu$ or $\vv$ without rendering the rank-$1$ CSDD for the LP problem infeasible  or causing its symmetry group  
to exclude a non-empty subset of $\Z_{\ell}$,  provided that this rank-$1$ CSDD is feasible in the first place.
Future research will involve  proving  Conjecture~\ref{conj:Dursun1}    with or without additional assumptions.

\section*{Acknowledgments}
  The views expressed in this article are those of the authors, and do not reflect the official policy or position of the United States Air Force, Department of Defense, or the U.S.~Government. \\

\bibliographystyle{plain}
\bibliography{bibliography}

\end{document}